\documentclass[a4paper,reqno]{amsart}%[pdflatex,sn-mathphys-num]{sn-jnl}%

\usepackage{hyperref}
\usepackage{enumerate}
\usepackage{comment}
\usepackage{amsthm}
\allowdisplaybreaks

\usepackage{amsmath}
\usepackage{comment}
\usepackage{amsfonts}
\usepackage{color}

\usepackage{graphicx}
\usepackage{subfig}
\usepackage{bm}
\usepackage{tikz}
\usepackage{amssymb}

\newtheorem{lemma}{Lemma}[section]
\newtheorem{proposition}[lemma]{Proposition}
\newtheorem{theorem}[lemma]{Theorem}
\newtheorem{corollary}[lemma]{Corollary}

\theoremstyle{definition}

\theoremstyle{remark}
\newtheorem{remark}[lemma]{Remark}
\newtheorem{notation}[lemma]{Notation}

\newcommand{\ffi}{\varphi}

\def\C{{\mathbb{C}}}

\def\N{{\mathbb{N}}}

\def\R{{\mathbb{R}}}
\def\T{{\mathbb{T}}}

\def\Z{{\mathbb{Z}}}

\def\bb{{\mathcal B}}

\def\cG{{\mathcal G}}

\newcommand{\norm}[1]{{\left\|{#1}\right\|}}

\newcommand{\abs}[1]{{\left|{#1}\right|}}
\newcommand{\scal}[1]{{\left\langle{#1}\right\rangle}}

\newcommand{\ud}{\,\mathrm{d}}

\usepackage{bbm}
\usepackage{mathtools}
\numberwithin{equation}{section}
\usepackage{graphicx} % Required for inserting images

\usepackage{orcidlink}

\title{Convergence of Hermite Expansions in Modulation Spaces }
\author[P. Jaming]{Philippe Jaming}
\address[P. J.]{Universit\'e de Bordeaux, CNRS, Bordeaux INP, IMB, UMR 5251, F-33400 Talence, France}
\email{philippe.jaming@math.u-bordeaux.fr,}
\author[M. Speckbacher]{Michael Speckbacher}
\address[M. S.]{ 
	Acoustics Research Institute\\ Austrian Academy of Sciences\\ Dominikanerbastei 16, 1010 Vienna,  Austria}
\email{michael.speckbacher@oeaw.ac.at}

\begin{document}
\begin{abstract}
The aim of this paper is to give an elementary proof that Hermite expensions of a function $f$ in the modulation
space $M^p(\R)$ converges to $f$ in $M^p(\R)$ when $1<p<\infty$ and may diverge when $p=1,\infty$.
The result was previously established for $1<p<+\infty$ by Garling and Wojtaszczyk \cite{GaWo}
and for $p=1,\infty$ by Lusky \cite{Lus} in an equivalent setting of Fock spaces by different methods.
Higher dimesional results are also considered.

In an appendix, we also establish upper bounds for the Zak transform of Hermite functions.
\end{abstract}

\maketitle

\section{Introduction}

The aim of this paper is to investigate   the convergence of Hermite series in the modulation spaces $M^p(\R^d)$, $1\le p< \infty$.
Let us  first recall that the $m$-th Hermite function is defined as
$$
h_n(t)=C_ne^{\pi t^2}\frac{\mbox{d}^n}{\mbox{d}t^n}\Big(e^{-2\pi t^2}\Big),\qquad t\in\R,\ n\in\N_0,
$$
where $C_n=(-1)^n \, {2^{1/4}}\big(\sqrt{n!}(2\pi)^{n}\big)^{-1/2}$ is chosen so that $\|h_n\|_{L^2(\R)}=1$. It is then well known that $(h_n)_{n\in\N_0}$ is an orthonormal basis
of eigenfunctions of both the harmonic oscillator and the Fourier transform
and thus play an important role in time-frequency analysis.

The focus of this paper is to investigate to which extent one can use Hermite functions as a generating system in time-frequency analysis.
To be more precise, from the orthonormality of the family $(h_n)_{n\in\N_0}$, it follows that 
any $f\in L^2(\R)$ can be written as a series of Hermite functions
$$ f=\sum_{n=0}^{\infty}\langle f,h_n\rangle h_n$$
where this series is $L^2$-convergent, that is
\begin{equation}
\label{eq:hermiteexpension}
\lim_{N\to\infty}\norm{f-\sum_{n=0}^N\scal{f,h_n}h_n}_{L^2(\R)}=0.
\end{equation}
It is then natural to ask whether other modes of convergence hold. In this direction,
%J.\,V.
Uspensky \cite{Us} established the pointwise convergence of the series already a century ago. Skipping forward  half a century, 
%R.
Askey and %S.
Wainger \cite{AW} proved  that in \eqref{eq:hermiteexpension}  one may replace the $L^2$-norm  by $L^p$-norms  if and only if $ {4}/{3}<p<4$.
Bochner–Riesz summability of these series outside the range ${4}/{3}<p<4$ was subsequently obtained by %S.
Thangavelu \cite{Th1,Th2}; we refer to his monograph \cite{Thbook} for a comprehensive account of the subject up to the late 1980s.

Further advances have been made since then. For instance, the first author, together with %A. 
Karoui and %S. 
Spektor \cite{JKS}, showed that Hermite expansions approximate almost time- and band-limited functions nearly as well as the prolate spheroidal wave functions. As a consequence, approximation rates for Sobolev functions by Hermite expansions were also established. Moreover, one may associate Sobolev spaces to the harmonic oscillator (see \cite{Thbook}); in this framework, Bongionnani and Torrea \cite{BT} proved convergence of Hermite series in such spaces.

Here we consider another natural family of function spaces in time-frequency analysis, namely the modulation spaces  which were introduced by Feichtinger in \cite{Fei,Fei3}.
When $z=(x,\omega)\in\R^{2d}$, and $f\in L^2(\R^d)$ we write $\pi(z) f(t)=e^{2i\pi \omega\cdot t}f(t-x)$
for the time-frequency shift of $f$.
Denoting by $\varphi(t)=2^{d/4}e^{-\pi |t|^2},\ t\in\R^d,$ the standard $d$-dimensional Gaussian
and $\mathcal{S}'(\R^d)$ the space of tempered distributions, the modulation spaces are then defined as
$$
M^p(\R^d)=\Big\{f\in \mathcal{S}'(\R^d)\,:\ \|f\|_{M^p(\R^d)}^p:=\int_{\R^{2d}}|\scal{f,\pi(z)\varphi}|^p\,\mbox{d}z<\infty\Big\}.
$$
When $p=1$ this gives the Feichtinger Algebra  \cite{Fei}.
For properties and applications of $M^p(\R^d)$ we refer to the surveys \cite{Fei2,Ber} as well as the book \cite{Gro}
and the references therein. Let us first reduce our analysis to $p=d=1$.  As $M^1(\R)\subset L^2(\R)$, if $f\in M^1(\R)$ then the corresponding Hermite series
converges in $L^2(\R)$, {\it i.e.}, \eqref{eq:hermiteexpension}
holds, but it seems unknown so far whether this convergence also holds
in the $M^1$-norm. To our knowledge, the only  related result is due to Janssen who
considered the spaces 
$$
C(\gamma)=\Big\{f\in L^2(\R)\,: \|f\|_{C(\gamma)}:=\sum_{n=0}^{\infty}|\langle f,h_n\rangle|(1+n)^\gamma<\infty\Big\}
$$
and showed that $C(1/4)\subset M^1(\R)\subset C(-1/4)$ but the space on the right hand side can not be made smaller
and the one on the left larger, at least within the $C(\gamma)$ scale. More precisely, he showed the following.

\begin{theorem}[Janssen \cite{Jan}]
There exists $A_\gamma$ such that, for every $f\in C(\gamma)$, $\|f\|_{M^1(\R)}\leq A_\gamma\|f\|_{C(\gamma)}$ if and only if $\gamma\geq 1/{4}$.

There exists $B_\gamma>0$ such that, for every $f\in M^1(\R)$, $B_\gamma\|f\|_{C(\gamma)}\leq B_\gamma\|f\|_{M^1(\R)}$ if and only if $\gamma\leq- {1}/{4}$.

Moreover,  $M^1(\R)\not\subset C(\gamma)$ when $\gamma>-1/4$, and $ C(\gamma)\not\subset M^1(\R) $ when $\gamma<1/4$.
\end{theorem}

As a consequence, this shows that
convergence in $M^1(\R)$ cannot be fully described in terms of the coefficients in the Hermite basis.
Here we prove that Hermite series do not converge in $M^1(\R^d)$ and therefore, by duality, also not in $M^\infty(\R)$
However, when $p$ is in the intermediate range, $1<p<\infty$,  the series converges in $M^p(\R^d)$.

Let us define the partial Hermite expansion operator via
$$
S_Nf:=\sum_{n=0}^N\langle f, h_n\rangle h_n.
$$
The main result of this paper is the following:

\begin{theorem}
\label{thm:hermiteins0}
    Let $1\leq p\leq \infty$. There exists a  constant $C=C_{p}$ such that for every $N\in\N$ 
    \begin{equation}\label{eq:conv-Mp}
        \big\|S_Nf\big\|_{M^p(\R)}\leq C\|f\|_{M^p(\R)},\qquad f\in M^p(\R),
    \end{equation}  
    if and only if $1<p<\infty$.
For $p=1$, there exist   constants $c,C>1$ such that 
\begin{equation}\label{eq:non-unif}
   \big\|S_N  \big\|_{M^1(\R)\to M^1(\R)}\ge C\log(c N) ,
\end{equation}
In particular, 
\begin{enumerate}
    \item for every $f\in M^p(\R),\ 1<p<\infty$, $\lim_{N\to \infty}S_Nf= f$
    in $M^p(\R^d)$
\item there exists $f\in M^1(\R) $ such that 
    $\|S_Nf\|_{M^2(\R)}\to\infty$ so that $S_Nf$ does not converge to $f$ in $M^1(\R)$.
    \end{enumerate}
\end{theorem}

We may reinterpret Theorem~\ref{thm:hermiteins0} as a result about the convergence of the Taylor expansion for certain spaces of holomorphic functions. To make that connection, we first recall that the 
\emph{Bargmann transform} of a function $f$ is given by
$$
\mathcal{B}f(z)=2^{1/4}\int_\R f(t)e^{2\pi zt-\pi t^2+\pi z^2/2}\ud t,\quad z\in \C, 
$$
and note that the Bargmann transform is an isometric isomorphism between the modulation spaces $M^p(\R)$ and the 
\emph{(Bargmann-)Fock spaces}  
$$
\mathcal{F}^p(\C):=\Big\{F:\C\to \C \text{ holomorphic: }\|F\|_{\mathcal{F}^p(\C)}:=\int_\C |F(z)|^pe^{-\pi p|z|^2/2}dz<\infty\Big\}.
$$
 The subsequent corollary then follows directly from Theorem~\ref{thm:hermiteins0} once we note that 
 $$
 \mathcal{B}h_n(z)=\left(\frac{\pi}{n!}\right)^{n/2}z^n.
 $$
\begin{corollary}[Garling-Wojtaszczyk \cite{GaWo}, Lusky \cite{Lus}]
Let $1\leq p<\infty.$
The partial Taylor expansion
$$
T_NF(z)=\sum_{n=0}^N \frac{F^{(n)}(0)}{n!} z^n 
$$
converges for every $F\in \mathcal{F}^p(\C)$ with respect to $\|\cdot\|_{\mathcal{F}^p(\C)}$ if and only if $1<p<\infty$
\end{corollary}

This corollary is not new and is of course  equivalent to Theorem \ref{thm:hermiteins0}.
Actually, the case $1<p<\infty$ can be found in \cite[Proposition 7]{GaWo} as well as \cite[Theorem 2.1]{Lus}
while the case $p=1$ is \cite[Theorem 2.3]{Lus}, though it is written in a slightly different language
and a more general form in \cite{Lus}. However, these results appear to have remained largely unnoticed within the time–frequency community\footnote{For instance, during discussions at a Quantum Harmonic Analysis workshop in Hannover in 2024 and at SampTA 2025 in Vienna, several colleagues mentioned this as an open problem.}. Beyond closing this knowledge gap, we also provide an alternative proof that is more direct and uses only elementary tools.
For the range $1<p<\infty$, our proof of \eqref{eq:conv-Mp} relies on \cite[Proposition~4.3]{multipliers} where a connection between boundedness of Hermite multipliers in $M^p(\R^d)$ and Fourier multipliers on $L^p(\T^d)$ was shown.
This idea is somewhat implicit in \cite{GaWo}. As a complementary result, this 
also allows us to study Bochner-Riesz means of Hermite series.
\begin{proposition}\label{prop:bochner-riesz}
    Let $1\leq p<\infty$ and $f\in M^p(\R)$. If $\alpha>0$, then the Bochner-Riesz means
    $$
    B_N^\alpha f:=\sum_{n=0}^N\left(1-\frac{|n|^2}{N^2}\right)^\alpha \langle f,h_n\rangle h_n
    $$
    converge to $f$
 in $M^p(\R)$ as $N\to \infty$. 
 \end{proposition}
 
In this introduction, we  restricted our analysis to modulation spaces of functions on the real line. However, our result extends in a straightforward way to expansions in terms of tensor products of Hermite functions in arbitrary dimensions. We will discuss this in Section~\ref{sec:higher-dim}.

\smallskip

The remaining of the paper is organized as follows: Section~\ref{proof:main} is devoted to the proof of the main result while
Section~\ref{sec:higher-dim} is devoted to higher dimensional extensions. 

For this arxiv version, we also include a bound on the Zak transform in Appendix~\ref{app}
which answers a question originally asked by H. Bahouri to one of us.
This is an elementary consequence of well known results and techniques in Gabor analysis.
The same result has since been proved by different methods by H. Bahouri and V. Fischer in
an upcoming paper.

\section{Proof of Main Results}\label{proof:main}

Our aim here is to prove Theorem \ref{thm:hermiteins0}.
Before we give the proof, we need to establish an auxiliary result on the $L^1$-norm of certain trigonometric polynomials that might be of independent interest. 

Let us define
\[
P_N(t,\phi)\;=\;\sum_{n=0}^N \frac{t^n  }{n!}e^{i\phi n}\qquad  t>0,\; \phi\in[0,2\pi].
\]

\begin{theorem}\label{th:P_N}
For every $N\in\N_0$ and $t\geq 1$ it holds
\begin{equation}\label{eq:1}
   \frac{\log(  t)\, t^N}{N!}-  \frac{30\, e^t}{\sqrt{t} }
%\le\int_0^{2\pi} \big|P_{\lceil t\rceil}(t,\phi)\big|\,d\phi
\leq  \int_0^{2\pi} \big|P_N(t,\phi)\big|\,\ud\phi\le \frac{\pi}{2}\, \frac{\log(\pi^2 \, t)\, t^N}{N!}+\frac{30\, e^t}{\sqrt{t}}
 . 
\end{equation}
\end{theorem}

\begin{proof}Note that \eqref{eq:1} holds trivially for $N=0$. We may therefore assume from now on that $N\geq 1$. We divide the proof
into several simple steps. We start with a preliminary step containing bounds
that we will need.

\medskip

\noindent{\em Step 1.} {\em Two simple bounds.}

\medskip

We will use the classical pointwise Stirling bound on the factorial
({\it see, e.g.,} \cite{Rob}) 
\begin{equation}
    \label{stirling}
\sqrt{2\pi n}\left(\frac{n}{e}\right)^n\leq n!\leq e\sqrt{n}\left(\frac{n}{e}\right)^n,
\end{equation}
valid for $n\geq 1$. 

  We will also need the following elementary geometric bound 
\[
\frac{2}{\pi}\, \phi\le |1-e^{i\phi}| = 2\sin\Big(\frac{\phi}{2}\Big) \le \phi,\qquad \phi\in[0,\pi].
\]
Equivalently, for \( \phi\in (0,\pi]\),
\begin{equation}\label{eq:basic}
\frac{1}{\phi}\le\frac{1}{|1-e^{i\phi}|} \le \frac{\pi}{2}\, \frac{1}{\phi}.
\end{equation}

\medskip

\noindent\textit{Step 2.} {\em Reformulation of the result.}

\medskip

Let us write
\[
p_n(t):= \frac{e^{-t}t^n}{n!},\qquad  n\in\N_0,
\]
so that \(\{p_n(t)\}_{n\ge0}\) are the Poisson\((t)\) probabilities and
%$S_N(t,\phi)=e^{t}P_N(t,\phi)$ with
\[
e^{-t}P_N(t,\phi)=\sum_{n=0}^N p_n(t)\,e^{in\phi}.
\]
Thus $e^t$ may be factored out in \eqref{eq:1} and
%Also, we will lower the sup over $N$ by choice $N=\lceil t\rceil$.
 it suffices to prove the inequality
\begin{equation}\label{eq:1refo}
 \frac{\log(  t )t^Ne^{-t}}{ N!}-\frac{30}{\sqrt{t}}  \le\int_0^{2\pi} e^{-t}\big|P_N(t,\phi)\big|\,\ud\phi
\le   \frac{\pi}{2}\frac{\log( \pi^2\, t )t^Ne^{-t}}{ N!}+\frac{30}{\sqrt{t}}.
\end{equation}
%for  absolute constants \(C,\alpha\) and all \(t\geq 1/\alpha\).  

\medskip

\noindent\textit{Step 3.} {\em Re-summation.}
\medskip

For any finite sequence \(q_0,q_1,\dots,q_N\) and \(z\neq1\) one has
\begin{equation}\label{eq:ibp2}
(1-z)^2 \sum_{n=0}^N q_n z^n
= \sum_{m=0}^{N+2} r_m z^m,
\end{equation}
with
\begin{equation}
    \label{eq:ibp2bis}
    \begin{cases}
        r_0:=q_0, &\\  r_1:=q_1-2q_0,&\\
        r_m := q_m - 2q_{m-1} + q_{m-2},&\mbox{for }2\leq m\leq N,\\
        r_{N+1}:=-2q_N+q_{N-1},&\\ r_{N+2}:=q_N.
    \end{cases}
\end{equation}
Applying and rearranging \eqref{eq:ibp2} with \(q_n=p_n(t)\) and \(z=e^{i\phi}\) one obtains
\begin{equation}\label{eq:key}
\sum_{n=0}^N p_n(t) e^{in\phi}
= \frac{1}{(1-e^{i\phi})^2}\sum_{n=0}^{N+2} r_n(t) e^{in\phi},\qquad \phi\notin 2\pi\Z.
\end{equation}
Next, we define 
$
\phi_0:={1}/{\sqrt{t}},
$
and 
$$
I_N(t):=\int_{\phi_0}^\pi\left|\sum_{n=0}^N p_n(t)e^{in\phi}\right|\,\mbox{d}\phi
=\int_{\phi_0}^\pi\frac{1}{\left|1-e^{i\phi}\right|^2}\,\left|\sum_{n=0}^{N+2} r_n(t) e^{in\phi}\right|\,\mbox{d}\phi.
$$
Since 
$$
\left|\sum_{n=0}^N p_n(t) e^{in\phi}\right|\le  \sum_{n=0}^N |p_n(t)|\leq1,
$$
it follows immediately that 
\[
2 I_N(t)\leq \int_0^{2\pi}\left|\sum_{n=0}^N p_n(t) e^{in\phi}\right|\,\mbox{d}\phi
\le 2 I_N(t)+\frac{2}{\sqrt{t}}.
\]
It is thus enough to bound $I_N$ from below and above. To do so, we will now decompose the sum 
$ \sum_{n=0}^{N+2} r_n(t) e^{in\phi}$
into two parts: the terms $n=N+1,N+2$ which will be shown to be dominating and the remaining sum
over $0\leq n\leq N$ which turns out to be a remainder term.

\medskip

\noindent\textit{Step 4.} {\em Estimation of the dominating term.}

\medskip

We first write
\begin{align*}
   r_{N+1}(t)+e^{i\phi}r_{N+2}(t) 
   &=p_{N-1}(t)-2p_N(t)+p_{N}(t)e^{i\phi}
 \\ & =\bigl( e^{i\phi}-1\bigr)p_N(t)+p_{N-1}(t)-p_N(t)
  % \\&=\bigl(1-e^{i\phi}\bigr)p_N(t)+\left(1-\frac{t}{N}\right)p_{N-1}(t)
   \\&=\bigl(e^{i\phi}-1\bigr)p_N(t)+ \frac{N-t}{t}  p_N(t).
\end{align*}
The absolute value of the second term can be bounded as follows
\begin{equation}\label{eq:3}
p_N(t)\frac{|N-{t}|}{t}\leq  \frac{1}{t}\cdot\sup_{x>0}\big\{p_N(x)|N-x| \big\}\le \frac{1}{t} ,\quad t>0.
\end{equation}
To prove this bound, first note that the supremum is attained at one of the two local maxima $x_{\pm}=N+\dfrac{1}{2}\pm\sqrt{N+\dfrac{1}{4}}$. With the help of \eqref{stirling} we may bound
\begin{align*}
\frac{e^{-x_{+}}x_+^{N}}{N!} {|N-x_+|} 
& \le     e^{-1/2-\sqrt{N+1/4} } \left(1+\frac{1}{ N}\left(\frac{1}{2}+\sqrt{N+\frac{1}{4}}\right)\right)^N  \frac{ \sqrt{N+ {1}/{4 }}+ {1}/{2 } }{\sqrt{2\pi N}}   \notag
\\
 & \le     e^{-1/2-\sqrt{N+1/4}+N\log(1+\delta) }   , 
\end{align*}
where $\delta= \big( {1}/{2}+\sqrt{N+{1}/{4}} \big)/N $ and we used that $ \big({ \sqrt{N+ {1}/{4 }}+ {1}/{2 } }\big)/{\sqrt{2\pi N}}  \le 1$  for $N\ge 1$. Since $\log(1+\delta)\leq \delta$ for every $\delta>-1$, we   infer that 
 $$
\frac{e^{-x_{+}}x_1^{N}}{N!} {|N-x_1|} \leq 1.
$$
For $x_-$ we argue similarly, only that this time we set $\delta=\big( {1}/{2}-\sqrt{N+{1}/{4}} \big)/N>-1. $ 
 
Using  
the lower and upper triangular inequalities then leads to
the bounds
\begin{equation}
    \label{eq:dominating}
 \bigl|1-e^{i\phi}\bigr|p_N(t) - \frac{1}{t} \le  {\bigl|r_{N+1}(t)+e^{i\phi}r_{N+2}(t)\bigr|
  }   \leq    \bigl|1-e^{i\phi}\bigr|p_N(t)+ \frac{1}{t} 
.
\end{equation}

\medskip

\noindent\textit{Step 5.} {\em Estimation of the remainder term.}

\medskip

Note that, for $2\leq n\leq N$, \(r_n(t)\)  
 is the second finite difference  of the Poisson probabilities 
\(p_n(t)=e^{-t}t^n/n!\). A short computation  yields for \(2\leq n\leq N\)
\[
r_n(t) := p_n(t)-2p_{n-1}(t)+p_{n-2}(t)  =   p_n(t)\, \frac{(n-t)^2-n}{t^2}.
\]
Hence
\[
\sum_{n=2}^N \big|r_n(t)\big|
\leq \frac{1}{t^2}\sum_{n\ge0} p_n(t) \big|(n-t)^2-n\big|
= \frac{1}{t^2}\,\mathbb{E}\big[\,|(X-t)^2-X|\,\big],
\]
where \(X\sim\mathrm{Poisson}(t)\). Using the triangle inequality and the moments of \(X\) we get the crude bound
\[
\mathbb{E}\big[\,|(X-t)^2-X|\,\big]\le \mathbb{E}\big[(X-t)^2+X\big]
= \operatorname{Var}(X)+\mathbb{E}[X] = t+t = 2t.
\]
Therefore we obtain the simple bound
\begin{equation}\label{eq:sum-delta2}
\sum_{n=2}^N\big|r_n(t)\big| \le \frac{2}{t}.
\end{equation}
On the other hand
$$
|r_0(t)+e^{i\phi}r_1(t)|\leq |r_0(t)|+|r_1(t)|\leq 3|p_0(t)|+|p_1(t)|\leq (3+t)e^{-t},
$$
which is easily seen to be less than $ {2}/{t}$ for $t\geq 1$. Together with \eqref{eq:sum-delta2},
this implies that
\begin{equation}
    \label{eq:sum-delta2bis}
    \sum_{n=2}^{N} |r_n(t)|+|r_0(t)+e^{i\phi}r_1(t)|\leq\dfrac{4}{t}.
\end{equation}

\medskip

 \noindent\textit{Step 6.} {\em The   bounds in \eqref{eq:1refo}.}

\medskip

From \eqref{eq:basic}, \eqref{eq:key}, \eqref{eq:dominating}, \eqref{eq:sum-delta2bis} and the reverse triangle inequality we obtain
\begin{align}
&\left|\sum_{n=0}^N p_n(t) e^{in\phi}\right|\notag
\\ &\hspace{1cm}\ge \frac{1}{|1-e^{i\phi}|^2}\left(\big|r_{N+1}(t)+e^{i\phi}r_{N+2}(t) \big|-\sum_{n=2}^{N} |r_n(t)|
-|r_0(t)+e^{i\phi}r_1(t)|\right)\notag
\\
&\hspace{1cm}\ge \frac{1}{|1-e^{i\phi}|}p_N(t)-\frac{1}{|1-e^{i\phi}|^2 }\frac{5}{t} \geq \frac{1}{\phi}p_N(t)-\frac{5 \pi^2}{4\phi^2 t} ,\qquad \phi_0\leq \phi\leq \pi.
\label{lowerboundphi}
\end{align}
Similarly, applying the triangle inequality instead of the reverse triangle inequality and the upper instead of the lower bound in \eqref{eq:basic} yields 
\begin{align}
&\left|\sum_{n=0}^N p_n(t) e^{in\phi}\right|\notag
\le  \frac{\pi}{2\phi}p_N(t)+\frac{5 \pi^2}{4\phi^2 t} ,\qquad \phi_0\leq \phi\leq \pi.
\notag
\end{align}

Integrating over $\phi$ thus gives
\begin{align*}
I_{N}(t)
&\ge  \int_{\phi_0}^{\pi}\left[ \frac{1}{\phi}p_N(t)-\frac{5\pi^2}{ 4\phi^2t}\right]\ud\phi
= \big(\log(\pi)-\log(\phi_0)\big)p_N(t)-\frac{5\pi^2}{4t} \left(\frac{1}{\phi_0}-\frac{1}{\pi}\right)
\\
&\ge  \frac{1}{2}\log(t)p_N(t)-\frac{5\pi^2}{4\sqrt{t}} \geq \frac{1}{2}\log(t)p_N(t)-\frac{15}{\sqrt{t}} ,
\end{align*}
where we substituted \(\phi_0=1/\sqrt{t}\) in the final step and removed useless positive terms. 
Moreover,
\begin{align*}
I_{N}(t)
&\le  \int_{\phi_0}^{\pi}\left[ \frac{\pi}{2\phi}p_N(t)+\frac{5\pi^2}{ 4\phi^2t}\right]\ud\phi
= \frac{\pi}{2}\big(\log(\pi)-\log(\phi_0)\big)p_N(t)+\frac{5\pi^2}{4t} \left(\frac{1}{\phi_0}-\frac{1}{\pi}\right)
\\
&\le  \frac{\pi}{4}\log(\pi^2 t)p_N(t)+\frac{5\pi^2}{4\sqrt{t}} ,
\end{align*}
 which completes the proof once we note that    $5\pi^2/2+2\leq 30.$

 \end{proof}

To prove  Theorem~\ref{thm:hermiteins0} in the range $1<p<\infty$ we need   the following result which connects boundedness of Hermite multipliers on $M^p(\R^d)$ with boundedness of Fourier multipliers with the same symbol on $L^p(\T^d)$.

\begin{proposition}\label{multipliers}\textbf{\emph{(\cite[Proposition~4.3]{multipliers})}}
Let $1\leq p<\infty$,  $m:\Z^d\to \C$ and $g:\T^d\to \C$. If the  \emph{ Fourier multiplier}   $A_m$ given by $\widehat{A_mg}(n):=m(n)\widehat{g}(n),\ n\in\Z^d$, defines a bounded operator on $L^p(\T^d)$, then 
the \emph{Hermite multiplier} with symbol $m|_{\N_0^d}$ given by
$$
H_m f=\sum_{n\in\N_0^d}m(n)\langle f,h_n\rangle h_n,
$$
is bounded on $M^p(\R^d)$ with
$$ 
\|H_m f\|_{M^p(\R^d)}\lesssim \|A_m\|_{L^p(\mathbb{T}^d)\to L^p(\mathbb{T}^d)}\|f\|_{M^{p}(\R^d)}
$$
\end{proposition}

 We are now ready to give the proof of 
Theorem~\ref{thm:hermiteins0}.
\begin{proof}[Proof of Theorem \ref{thm:hermiteins0}]
%The second part of the theorem is of course a direct consequence of the first part, the Banach-Steinhaus Theorem and the density of finite linear combinations of Hermite functions in $S_0(\R)$.
Let us identify the phase space $\R^2$ with $\C$ via $z=(x,\omega)\in\R^2\sim z=x+i\omega\in\C.$ By the Laguerre connection (see, e.g., \cite[Theorem~1.104]{folland}), one has
$$
\langle  \pi(z)h_0,h_n\rangle =e^{\pi i x\xi-\pi |z|^2/2}\sqrt{\frac{\pi^n}{n!}} {z}^n,\qquad z\in\C,\ n\in\N_0.
$$
Therefore, for $w=y+i\eta$
\begin{equation}
\label{eq:idenherm}
    \sum_{n=0}^N \langle  \pi(z)h_0,h_n\rangle\langle h_n,\pi(w)h_0\rangle
    =e^{ \pi i (x\xi-y\eta)}e^{-\pi (|z|^2+|w|^2)/2}\sum_{n=0}^N \frac{(\pi z\overline{w})^n}{n!}
    ,
    \end{equation}
  \begin{comment}  Let us introduce the notation $S_Nf=\sum_{n=0}^N\langle f,h_n\rangle h_n$. 
    Using the reproducing formula for the short-time Fourier transform \textcolor{red}{(ref)}, we may rewrite the $M^q$-norm of $S_Nf$ as follows
       \begin{align*}
\|S_Nf\|_{M^q(\R)}^q
&=\int_{\C}\left|\sum_{n=0}^N \langle f,h_n\rangle \langle h_n,\pi(z)h_0\rangle \right|^q\,\mbox{d}z
        \\
&=\int_{\C}\left|  \int_{\C}\langle f,\pi(w)h_0\rangle \sum_{n=0}^N\langle \pi(w)h_0,h_n\rangle \langle h_n,\pi(z)h_0\rangle\,\mbox{d}w\right|^q\,\mbox{d}z
        \\
&=\int_{\C}\left|  \int_{\C}F(w)\sum_{n=0}^N\langle \pi(w)h_0,h_n\rangle \langle h_n,\pi(z)h_0\rangle\,\mbox{d}w\right|^q\,\mbox{d}z
\\
&=\|T_NF\|_{L^q(\C)},
\end{align*}
where $F(w)=\langle \pi(w)h_0,h_n\rangle ,$ and $T_N:L^p(\C)\to L^q(\C)$ is the integral operator with integral kernel
$$
K_N(z,w)=\sum_{n=0}^N\langle \pi(w)h_0,h_n\rangle \langle h_n,\pi(z)h_0\rangle
$$
Therefore, by Young's inequality for integral operators (and the symmetry of $K_N$) we establish \eqref{eq:conv-Mp-Mq} once we show that for every $r>1$ there is a constant $C=C_r$ such that  
\begin{equation}\label{eq:Lr}
\sup_{N\in\N}\sup_{z\in\C}\|K_N(z,\cdot)\|_{L^r(\C)}\leq C.
\end{equation}
\end{comment}
If we set $|z|=r$ and $|w|=t$, we obtain by changing to polar coordinates
 \begin{align*}
  \big\|S_N(\pi(z)h_0)\big\|_{M^1(\R)}  & =\int_{\C}\left|\sum_{n=0}^N\langle \pi(z)h_0,h_n\rangle \langle h_n,\pi(w)h_0\rangle\right|\,\mbox{d}w
 \\& =  \int_{\C}e^{-\pi ( |z|^2+|w|^2)/2}\left|\sum_{n=0}^N \frac{(\pi w\overline{z})^n}{n!}\right|\,\mbox{d}w 
  \\& =  \int_{0}^\infty te^{-\pi ( r^2+t^2)/2}\int_0^{2\pi}\left|\sum_{n=0}^N \frac{(\pi rt)^n}{n!}e^{i\phi n}\right|\,\mbox{d}\phi\,\mbox{d}t
    \\& =  \int_{0}^\infty te^{-\pi ( r^2+t^2)/2}\int_0^{2\pi}\left|P_N(\pi r t,\phi)\right|\,\mbox{d}\phi\,\mbox{d}t.
 \end{align*}
Now we set $z_N=\sqrt{N/\pi}$ and assume $N\geq 3$. Applying the change of variables $\rho=  \sqrt{\pi}t $ and   Theorem~\ref{th:P_N} then yields
  \begin{align*}
  \big\|S_N(\pi(z_N)h_0)&\big\|_{M^1(\R)}  =\frac{1}{\pi}\int_{0}^{\infty}\rho e^{-(\rho^2+N)/2}
\int_0^{2\pi}\left|P_N\big(\sqrt{N}\rho,\phi\big)\right|\,\mbox{d}\phi\,\mbox{d}\rho 
\\ &\ge\frac{1}{\pi}\int_{\sqrt{N}-1}^{\sqrt{N}+1}\rho e^{-(\rho^2+N)/2}
\int_0^{2\pi}\left|P_N\big(\sqrt{N}\rho,\phi\big)\right|\,\mbox{d}\phi\,\mbox{d}\rho
\\ &\ge \frac{1}{\pi}\int_{\sqrt{N}-1}^{\sqrt{N}+1}\rho e^{-(\rho^2+N)/2}
\left(\frac{\log(\rho\sqrt{N})(\rho\sqrt{N})^N}{N!}-\frac{30\,  e^{\rho\sqrt{N}}}{(\rho\sqrt{N})^{1/2}}\right)\mbox{d}\rho
.
 \end{align*}
Note that we were allowed to apply Theorem~\ref{th:P_N} since $\rho\sqrt{N}\geq 1$ for every $\rho\in[\sqrt{N}-1,\sqrt{N}+1]$ and $N\geq 3.$
 
% So let us establish \eqref{eq:Lr} and \eqref{eq:diverge}.  Changing to polar coordinates and applying \eqref{eq:2}, we may  write
 %   \begin{align*}
%\sup_{|z|\geq 1}\|K_N(z,\cdot)\|_{L^r(\C)}^r&= \sup_{|z|\geq 1}\int_{\C}e^{-\pi r( |z|^2+|w|^2)/2}\left|\sum_{n=0}^N \frac{(\pi w\overline{z})^n}{n!}\right|^r\,\mbox{d}w 
%\\
%&=\sup_{t\geq 1}\int_{0}^{\infty}\rho e^{-\pi r(\rho^2+t^2)/2}
%\int_0^{2\pi}\left|\sum_{n=0}^N \frac{(\pi \rho t)^n}{n!}e^{in\phi}\right|^r\,\mbox{d}\phi\,\mbox{d}\rho 
%\\ &\leq C_r\sup_{t\geq 1}\int_{0}^{\infty}\rho e^{-\pi r(\rho^2+t^2)/2}
%\frac{e^{r\pi \rho t}}{\sqrt{\pi \rho t}}\mbox{d}\rho
%\\&=\frac{C_r}{\sqrt{\pi}}\sup_{t\geq 1}\int_{0}^{\infty}\sqrt{\frac{\rho}{t}} e^{-\pi r(\rho-t)^2/2}d\rho
%\\&=\frac{C_r}{\sqrt{\pi}}\sup_{t\geq 1}\int_{-t}^{\infty}\left( \frac{\rho}{t}+1\right)^{1/2}\, e^{-\pi r \rho^2/2}d\rho
%<\infty,
%    \end{align*}
 % where we used the rotational invariance of the integral in the second identity.
 % If $|z|\leq 1$, then 
  %  \begin{align*}
%\sup_{|z|\leq 1}\|K_N(z,\cdot)\|_{L^r(\C)}^r&\leq 2\pi\sup_{t\leq 1}\int_{0}^{\infty}\rho e^{-\pi r(\rho^2+t^2)/2}
% \left(\sum_{n=0}^N \frac{(\pi \rho t)^n}{n!} \right)^r\,\mbox{d}\rho 
%\\ &\leq 2\pi\sup_{t\leq 1}\int_{0}^{\infty}\rho e^{-\pi r(\rho^2+t^2)/2}
% e^{r\pi \rho t} \mbox{d}\rho
%\\&\leq 2\pi \int_{0}^{\infty}\rho e^{-\pi r \rho^2/2+\pi r\rho}d\rho
%<\infty.
 %   \end{align*}
 %   This establishes \eqref{eq:Lr}. 
 
Next we  prove the following auxiliary inequality: There exists $C>0$ such that for every $N\geq 3$ 
 \begin{equation}\label{eq:aux-inequ}
 \frac{e^{-t}t^{N+1/2}}{N!}\geq \frac{1}{4e},\qquad  |t-N|\leq\sqrt{N}.
 \end{equation}
 Note that the map $t\mapsto e^{-t}t^{N+1/2}$  has only one local maximum at  $t^\ast=N+1/2$. Therefore, to verify \eqref{eq:aux-inequ} it suffices to check the inequality for $t_{\pm}=N\pm\sqrt{N}$.
Using Stirling's approximation \eqref{stirling} we thus deduce  for $t_+$
$$
 \frac{e^{-N- \sqrt{N}}(N+\sqrt{N})^{N+1/2}}{N!}\geq e^{-\sqrt{N}-1}\left(1+ 1/\sqrt{N}\right)^{N +1/2} \geq \frac{1}{2e}.
$$
Similarly, for $t_-$
$$
 \frac{e^{-N+ \sqrt{N}}(N-\sqrt{N})^{N+1/2}}{N!}\geq e^{\sqrt{N}-1}\left(1- 1/\sqrt{N}\right)^{N +1/2} \geq \frac{1}{4e}.
$$
Consequently, 
    \begin{align*}
 \big\|S_N(\pi(z_N)h_0)\big\|_{M^1(\R)}  &\ge \frac{1}{\pi}\int_{\sqrt{N}-1}^{\sqrt{N}+1}\rho e^{-(\rho^2+N)/2}
\left(\frac{\log(\rho\sqrt{N}) }{4e}- {30  } \right)\frac{e^{\rho\sqrt{N}}}{(\rho\sqrt{N})^{1/2}}\mbox{d}\rho
\\
&\geq \frac{1}{4e\pi}\int_{\sqrt{N}-1}^{\sqrt{N}+1}\sqrt{\frac{\rho}{\sqrt{N}}} e^{-(\rho-\sqrt{N})^2/2}
 \log\left(\frac{\rho\sqrt{N}}{120e}\right)\mbox{d}\rho
\\ 
&\ge  \frac{1}{4e\pi}  \sqrt{1-\frac{1}{\sqrt{N}}}\log\left(\frac{N-\sqrt{N}}{120e}\right)\int_{-1}^{1}  e^{-t^2/2}
\mbox{d}t
\\&   \ge \frac{1}{4e^{3/2}\pi}  \log\left(\frac{N}{360e}\right) \gtrsim  {\log(cN)},
    \end{align*}
where we used that $N-\sqrt{N}\geq N/3$ for $N\geq 3$.  
Since $\|\pi(z)h_0\|_{M^1(\R)}=\|h_0\|_{M^1(\R)}$ for every $z\in\C$ it follows that $$\|S_N\|_{M^1(\R)\to M^1(\R)}\gtrsim \log(cN), $$
which  concludes the proof of \eqref{eq:non-unif}.
 The existence of an element $f\in M^1(\R)$ whose Hermite expansion does not converge in $M^1(\R)$ then follows from the Banach-Steinhaus theorem.

To prove the uniform boundedness of $S_N$ on $M^p(\R)$, $1<p<\infty,$ we first observe that the partial Fourier sum 
\begin{equation}\label{eq:partial-Fourier}
\sum_{-N\leq n\leq N}\widehat{f}(n) e^{2\pi i \xi\cdot n}, \qquad \xi\in [0,1],
\end{equation}
can be written as a Fourier multiplier $A_{m_N}$ with symbol $$m_N(n)=\chi_{[-N,N]}(n),\qquad n\in\Z.$$
It is well-known by the Marcel Riesz inequality that the partial Fourier sums  in \eqref{eq:partial-Fourier} are uniformly bounded on $L^p(\mathbb{T})$ which implies by Proposition~\ref{multipliers} that the family of Hermite multipliers $H_{m_N}$ are uniformly norm bounded on $M^p(\R)$. But $H_{m_N}$ is nothing else than $S_N$, which concludes the proof.
\end{proof}

\section{Convergence in higher dimensions}\label{sec:higher-dim}

In higher dimensions we   consider the tensor product Hermite functions, defined by \[h_n(t)=h_{(n_1,n_2,\dots,n_d)}(t_1,t_2,\dots,t_d)=\prod_{j=1}^d h_{n_j}(t_j),\qquad n\in\N_0^d,\ t\in\R^d.\]
Let us define the partial Hermite expansion operator via
$$
S_Nf:=\sum_{0\leq n_1,...,n_d\leq N}\langle f, h_n\rangle h_n.
$$

\begin{corollary}
\label{thm:hermiteins0-d}
    Let $1\leq p\leq \infty$. There exists a  constant $C=C_{p,d}$ such that for every $N\in\N$ 
    \begin{equation}\label{eq:conv-Mp-d}
        \big\|S_Nf\big\|_{M^p(\R^d)}\leq C\|f\|_{M^p(\R^d)},\qquad f\in M^p(\R^d),
    \end{equation}  
    if and only if $1<p<\infty$.
For $p=1$, there exists  a constant $c>1$ such that 
\begin{equation}\label{eq:non-unif-d}
   \big\|S_N  \big\|_{M^1(\R^d)\to M^1(\R^d)}\gtrsim_d \log(c N)^d ,
\end{equation}
In particular, for every $f\in M^p(\R^d),\ 1<p<\infty$,     $
    \lim_{N\to \infty}S_Nf= f
    $
    in $M^p(\R^d)$ and there exists $f\in M^1(\R^d) $ such that 
    $
    \lim_{N\to \infty}S_Nf\neq f
    $
    in $M^1(\R^d)$.
\end{corollary}

\begin{proof}
      For $d\geq 2$ one can simply consider the sequence $z_N=(\sqrt{N},...,\sqrt{N},0,...,0)\in \R^{2d}$ and observe that 
 \begin{align*}
  S_N\big(\pi(z_N)h_{(0,...,0)}\big)&=\sum_{0\leq n_1,...,n_d\leq N}  \prod_{k=1}^d \big\langle \pi\big(\sqrt{N},0\big)h_0,h_{n_k}\big\rangle h_{(n_1,...,n_d)}
  \\ &=\bigotimes_{k=1}^d S_N\big(\pi(\sqrt{N},0)h_0\big).
\end{align*}
Therefore, we obtain \eqref{eq:non-unif} as 
\begin{align*}
\big\|  S_N\big(\pi(z_N)h_{(0,...,0)}\big)\big\|_{M^1(\R^d)}&=  \big\|S_N\big(\pi\big(\sqrt{N},0\big)h_{0}\big)\big\|_{M^1(\R)}^d \\ &\gtrsim \log(CN)^d\|h_0\|_{M^1(\R)}^d=\log(CN)^d\|h_{(0,\dots,0)}\|_{M^1(\R^d)}.
\end{align*}
To prove the uniform boundedness of $S_N$ on $M^p(\R^d)$, $1<p<\infty,$ we first observe that the partial Fourier sum 
\begin{equation}\label{eq:partial-Fourier-d}
\sum_{-N\leq n_1,...,n_d\leq N}\widehat{f}(n) e^{2\pi i \xi\cdot n}, \qquad \xi\in [0,1]^d,
\end{equation}
can be written as a Fourier multiplier $A_{m_N}$ with symbol 
$$
m_N(n)=\chi_{[-N,N]}(n_1)\cdot ...\cdot\chi_{[-N,N]}(n_d),\qquad n\in\Z^d.
$$
It is well-known that the partial Fourier sums  in \eqref{eq:partial-Fourier-d} are uniformly bounded on $L^p(\mathbb{T}^d)$ (see, e.g., \cite[Section~4.1]{grafakos}) which implies by Proposition~\ref{multipliers} that the family of Hermite multipliers $H_{m_N}$ are uniformly norm bounded on $M^p(\R^d)$. But $H_{m_N}$ is nothing else than $S_N$, which concludes the proof.
\end{proof}

Our proof of \eqref{eq:conv-Mp} relies on \cite[Proposition~4.3]{multipliers} where a connection between boundedness of Hermite multipliers in $M^p(\R^d)$ and Fourier multipliers on $L^p(\T^d)$ was established. This  allows us to  also study Bochner-Riesz means of Hermite series.

Proposition~\ref{prop:bochner-riesz} is a special case of the subsequent result.

\begin{proposition}\label{prop:bochner-riesz-d}
    Let $1\leq p<\infty$, $f\in M^p(\R^d)$ and $(x)_+:=x\cdot\chi_{[0,\infty)}(x)$. If $\alpha>(d-1)\abs{\dfrac{1}{2}-\dfrac{1}{p}}$, then the Bochner-Riesz means
    $$
    B_R^\alpha f:=\sum_{{n\in\N_0^d} }\left(1-\frac{|n|^2}{R^2}\right)^\alpha_+ \langle f,h_n\rangle h_n
    $$
    converge to $f$
 in $M^p(\R^d)$ as $R\to \infty$. 
 \end{proposition} 

 \begin{proof}%[Proof of Proposition~\ref{prop:bochner-riesz}]
    Since the family of Fourier multipliers $A_{m_R}$ with $m_R(n)=(1-|n|^2/R^2)_+^\alpha$ is uniformly bounded  on $L^p(\T^d)$ whenever $\alpha>(d-1)\abs{\dfrac{1}{2}-\dfrac{1}{p}}$ (see, e.g., \cite[Proposition~4.1.9]{grafakos}), the result follows immediately from Proposition~\ref{multipliers}.
\end{proof}

\begin{remark}
Since \cite[Proposition~4.3]{multipliers} provides only a sufficient condition for the boundedness of Hermite multipliers, it remains an open problem whether certain symbols that generate unbounded Fourier multipliers also yield unbounded Hermite multipliers. A natural example to consider is Fefferman’s famous counterexample for the ball multiplier \cite{Feff}. This leads to the following question: for which range of 
$p$ does the sequence of spherical partial sums in the Hermite expansion,
$$
\overset{\circ}{S}_Nf:=\sum_{{n\in\N_0^d},\, {|n|\le N}}  \langle f,h_n\rangle h_n
$$
converge in $M^p(\R^d)$?
\end{remark}

\appendix

\section{Upper frame and Bessel bounds of Gabor systems associated to Hermite Functions}\label{app}

\begin{center}
    This appendix is not for final publication.
\end{center}

In this appendix, we answer a question raised to one of us by H. Bahouri, namely an upper bound
of the Zak transform of Hermite functions. While this result seems not written as such, it follows
from standard results and techniques in Gabor analysis. We have decided to write this appendix for the arxiv version of
the paper for eventual reference.

\subsection{Statement of results}

\begin{notation}
For $\Lambda\subset\R^2$ a lattice, we denote by 
    $$
    \text{rel}(\Lambda)=\max_{x\in\R^2}|\Lambda\cap(x+[0,1)^2)|
    $$
    the largest number of points of $\Lambda$ in a unit square. Note that   $\text{rel}(\Z^2)=1$.

The Gabor system generated by $h_n$ and $\Lambda$ is then defined by
    $$
    \cG(h_n,\Lambda)=\{\pi(\lambda) h_n\,:\ \lambda\in\Lambda\}.
    $$
\end{notation}

The aim of this section is to obtain an estimate of the upper frame bound and the upper Bessel bounds
of $\cG(h_n,\Lambda)$.

\begin{theorem}\label{thm:app}
    There exists a universal constant $C$ such that, for every lattice $\Lambda\subset \R^2$ and every $n\in\N_0$,
    \begin{enumerate}
        \item for every $f\in L^2(\R)$
        $$
        \sum_{\lambda\in\Lambda}\big|\scal{f,\pi(\lambda) h_n}\big|^2\leq C\,  \emph{rel}(\Lambda) (n+1)^{1/2}\|f\|^2_{2};
        $$
        \item for every $a=(a_\lambda)_{\lambda\in\Lambda}\in\ell^2(\Lambda)$,
        $$
        \norm{\sum_{\lambda\in\Lambda} a_\lambda\pi(\lambda) h_n}_{L^2(\R)}^2\leq C\, \emph{rel}(\Lambda) (n+1)^{1/2}\|a\|_{2}^2.
        $$
    \end{enumerate}
\end{theorem}

Recall that the Zak transform of a function $f\in L^2(\R)$ is defined as
$$
Zf(x,y)=\sum_{k\in\Z}f(x+k)e^{2i\pi kx}.
$$
It is well-known that upper and lower bounds for the Zak transform of the window function $g$ are related to the frame bounds of the corresponding
Gabor system $\mathcal{G}(g,\Z^2)$. In particular,  from \cite[Corollary 8.3.2]{Gro}
$$
\sup_{\|f\|=1}\sum_{\lambda\in\Z^2}\big|\scal{f,\pi(\lambda)g}\big|^2=\|Zg\|_{L^\infty([0,1]^2)},
$$
we may  directly deduce the following bound\footnote{This answers a question posed by H. Bahouri to one of us, and we thank her for sparking our interest in the problem.}:

\begin{corollary}
We have
$$
\|Zh_n\|_{L^\infty([0,1]^2)}\lesssim (n+1)^{1/4}.
$$
\end{corollary}

\subsection{The proof of Theorem~\ref{thm:app}}

We will use the basic properties of the Bargman transform that we recalled at the end of the introduction.
The key element of the proof is the following simple estimate of the $M^1(\R)$-norm of Hermite functions
that was already proved by A.E.J.M. Janssen:

\begin{lemma}\label{lem:m1normhermite}
    When $n\to\infty$,
    $$
    \|h_n\|_{M^1(\R)}=(2^3\pi)^{1/4}\,n^{1/4}\bigl(1+O(n^{-1})\bigr).
    $$
\end{lemma}

\begin{proof}%[Proof of Lemma \ref{lem:m1normhermite}]
Using properties of the Gabor transform and integration in polar coordinates, we have
\begin{align*}
    \|h_n\|_{M^1(\R)}&=\int_{\C}|\bb[h_n](z)|e^{-\pi|z|^2/2}\,\mbox{d}z
    =\frac{1}{\sqrt{n!}}\int_{\C}(\pi^{1/2}|z|)^n e^{-\pi|z|^2/2}\,\mbox{d}z\\
    &=\frac{2\pi}{\sqrt{n!}}\int_0^{\infty}(\pi^{1/2}r)^ne^{-\pi r^2/2}\,r\mbox{d}r\\
    %&=&{\color{blue}\frac{2^{\frac{m+3}{2}}\sqrt{\pi}}{\sqrt{m!}}\int_0^{+\infty}\left(\sqrt{\dfrac{\pi r^2}{2}}\right)^{m+1}e^{-\pi r^2/2}\,\mbox{d}r}\\
    %&=&{\color{blue}\sqrt{\frac{2}{\pi}}\frac{2^{\frac{m+3}{2}}\sqrt{\pi}}{\sqrt{m!}}\int_0^{+\infty}\left(s\right)^{\frac{m+1}{2}}e^{-s}\,\frac{\mbox{d}s}{2s^{1/2}}}\\
    &=2\frac{2^{\frac{n}{2}}}{\sqrt{n!}}\int_0^{\infty}s^{\frac{n}{2}+1}e^{-s}\,\frac{\mbox{d}s}{s}=2\frac{2^{\frac{n}{2}}\Gamma\left(\frac{n}{2}+1\right)}{\Gamma(n+1)^\frac{1}{2}}
\end{align*}
with the change of variable $r=\sqrt{\dfrac{2}{\pi}s}$.
Next, we apply Stirling's formula
$$
\Gamma(x+1)=\sqrt{2\pi}e^{-x}x^{x+\frac{1}{2}}\bigl(1+O(x^{-1})\bigr)
$$
to obtain
\begin{eqnarray*}
    \|h_n\|_{M^1(\R)}%&=&{\color{blue}2(2\pi)^{1/4}\frac{2^{\frac{m}{2}}e^{-\frac{m}{2}}\left(\frac{m}{2}\right)^{\frac{m+1}{2}}}{\sqrt{e^{-m}}m^{\frac{1}{2}\left(m+\frac{1}{2}\right) }}\bigl(1+O(m^{-1})\bigr)}\\
    &=&2\left(\frac{\pi}{2}n\right)^{1/4}\bigl(1+O(n^{-1})\bigr)
\end{eqnarray*}
as claimed.
\end{proof}

Theorem~\ref{thm:app} is then a direct consequence of the following lemma \cite[Lemma 2.1]{GL2}:

\begin{lemma}[Gr\"ochenig-Lyubarskii]
    If $\|\ffi\|_{M^1}<\infty$ then
        \begin{enumerate}
        \item for every $f\in L^2(\R)$
        $$
        \sum_{\lambda\in\Lambda} \big|\scal{f,\pi(\lambda) \ffi}\big|^2\leq C \, \emph{rel}(\Lambda) \|\ffi\|_{M^1}^2\|f\|^2_{L^2(\R)};
        $$
        \item for every $a=(a_\lambda)_{\lambda\in\Lambda}\in\ell^2(\Lambda)$,
        $$
        \norm{\sum_{\lambda\in\Lambda} a_\lambda\pi(\lambda )\ffi}^2_{L^2(\R)}\leq C\, \emph{rel}(\Lambda) \|\ffi\|_{M^1}^2\|a\|_{2}^2.
        $$
    \end{enumerate}
\end{lemma}

%\section{Data availability}
%No data has been generated or analyzed during this study.

\section*{Acknowledgements}

The authors have no relevant financial or non-financial interests to disclose.

The authors thank H. Bahouri, H. Feichtinger and K. Gr\"ochenig for pointing out to the questions
addressed here and for valuable conversation on the subject.

This work was partially supported by the French National Research Agency (ANR) under contract number ANR-24-CE40-5470, and the Austrian Science Fund (FWF) via the project 10.55776/PAT1384824.

For open
access purposes, the authors have applied a CC BY public copyright license to any author-accepted manuscript
version arising from this submission.

\end{document}